\documentclass[11pt,a4paper]{article}
\usepackage[utf8]{inputenc}
\usepackage{amsthm}
\usepackage{amsmath}
\usepackage{amssymb}
\usepackage{enumerate}
\usepackage{tikz}
\usetikzlibrary{positioning,chains,fit,shapes,calc,shapes.geometric}
\usepackage{geometry}
\usepackage{hyperref}
\usetikzlibrary{decorations.markings}

\newtheorem{theorem}{Theorem}
\newtheorem{lemma}{Lemma}
\newtheorem{question}{Question}
\newtheorem{claim}{Claim}
\newtheorem{observation}{Observation}

\newtheorem{corollary}{Corollary}
\newtheorem{conjecture}{Conjecture}

\begin{document}

\tikzset{->-/.style={decoration={
  markings,
  mark=at position .5 with {\arrow{>}}},postaction={decorate}}}

\title{Independent Dominating Sets in Directed Graphs}
\author{Adam Blumenthal\thanks{Department of Mathematics, Iowa State University, Ames, IA, E-mail: {\tt ablument@iastate.edu}}}
\date{\vspace{-5ex}}

\maketitle

\begin{abstract}
    In this paper, we study independent domination in directed graphs, which was recently introduced by Cary, Cary, and Prabhu. 
    We provide a short, algorithmic proof that all directed acyclic graphs contain an independent dominating set.
    Using linear algebraic tools, we prove that any strongly connected graph with even period has at least two independent dominating sets, generalizing several of the results of Cary, Cary, and Prabhu.
    We show that determining the period of the graph is not sufficient to determine the existence of an independent dominating set by constructing a few examples of infinite families of graphs.
    We also show that the direct analogue of Vizing's Conjecture does not hold for independent domination number in directed graphs by providing two infinite families of graphs.
    
    We also initialize the study of time complexity for independent domination in directed graphs, proving that the existence of an independent dominating set in directed acyclic graphs and strongly connected graphs with even period are in the time complexity class $P$.
    We also provide an algorithm for determining existence of an independent dominating set for digraphs with period greater than $1$.

    \begin{flushleft}\textbf{Keywords:} Independent Sets, Dominating Sets, Independent Dominating Sets, Directed Graphs\end{flushleft}
    
\end{abstract}

\section{Introduction}

Both the dominating set problem and independent set problem have been studied extensively in graphs. 
Independence has been widely studied for its relation to chromatic number, while domination has a deep relationship with communication in networks.
The study of sets that are both independent and dominating (or independent dominating sets) has history dating back to 1862, when de Jaenisch \cite{DJ} asked for the minimum number of non-attacking queens which can be placed on a chessboard such that every other square is threatened.
We note also that both independence and domination are classic examples of $NP$-complete problems, as is finding the smallest independent dominating set \cite{GareyJohnson}. 
It has been proven that determining the minimum size of an independent dominating set is $NP$-complete even in restricted families including bipartite graphs or line graphs \cite{Manlove,YannakakisGavril,CorneilPerl}.
The minimum size of a dominating set is used as a measure of efficiency of backbones for communications networks, and independent domination can be used for communication networks in which interference or fading can occur.
Further results include Nordhaus-Gaddum type results \cite{GoddardHenning,GHLS12}, and results for claw-free graphs \cite{AllanLaskar}, as well as random graphs \cite{DuckworthWormald}. 
For a thorough survey of the history and results in independent domination theory, we direct the reader to the paper \cite{GoddardHenningSurvey}.

In directed graphs independence is no different from the question in undirected graphs. 
On the other hand dominating sets are drastically affected by direction.
There is a long history of dominating set problems in directed graphs, but frequently they are restricted to certain families of graphs.
In particular, domination in tournaments has been studied for decades, including questions of Erd\H{o}s \cite{Erdos} and Gy\'arf\'as \cite{Gyarfas}.
More recently, Caro and Henning \cite{CaroHenning} continued the study of dominating set theory in directed graphs, providing some general bounds as well as relating the directed domination number to the independence number in bipartite graphs.

In 2019, Cary, Cary, and Prabhu \cite{CaryPrabhu} introduced independent domination in directed graphs.
This problem has relations to finding communication points for information transmission, particularly when information can only be sent in one direction at a time in a network.
As such, they explore the parameter with respect to oriented graphs since they correlate to ad-hoc networks \cite{DEHH}.

We define a \emph{directed graph} $D = (V,A)$ to be an ordered pair, where $V$ is a set called vertices ($V(G)$) and $A$ is a set of pairs of vertices called the edge set or arc set ($A(G)$).
A set of vertices $S$ to be \emph{independent} in a directed graph $D$ if there does not exist $u, v \in S$ such that $(u,v)$ is an arc in $D$.
A set of vertices $S$ to be \emph{dominating} in a directed graph $D$ if for every $v \in V(G) \setminus S$ there exists some $v \in S$ such that $(u,v) \in A(D)$.
A set of vertices $S$ to be \emph{independent dominating} in a directed graph $D$ if $S$ is both independent and dominating.

Cary, Cary, and Prabhu \cite{CaryPrabhu} provide results on certain families of graphs including orientations of bipartite graphs and cycles as well as directed acyclic graphs.
In this paper, we extend the study of independent domination into directed graphs which allow antiparallel edges, noting that parallel edges do not affect independent domination in directed graphs.
All directed graphhs will be assumed to be finite.
We will provide a result which generalizes several of the results of the previous paper, namely determining idomatic number for directed graphs with certain periods.
We additionally provide some alternative, algorithmically focused, proofs of similar results to Cary, Cary, and Prabhu.
We also begin the study of time complexity of independent dominating sets, showing that determining the smallest size of an independent dominating set in a directed graph in is $NP$-complete and providing an algorithm which answers this question in $O(1.26^n)$ time when the period of the graph is not one.
Cary, Cary, and Prabhu also introduce the concept of idomatic number of a graph $G$, and explore the parameter in some families of graphs.
In the conclusion, we suggest possible avenues for furthering the theory of independent domination in directed graphs.

\begin{section}{A Greedy Heuristic}
In this section we will provide a simple heuristic for finding an independent dominating set, which gives some short alternative proofs to those given in \cite{CaryPrabhu}. Our goal throughout this section is to provide a tool for determining the existence of an independent dominating set in a directed graph, with the goal of classifying graphs which contain no independent dominating set which we call \emph{independent dominating set-free (IDS-free)}.

Note that in undirected graphs, there always exists an independent dominating set which can be made by greedily adding vertices until we reach a maximal independent set.
In directed graphs, this is not the case. 
Notice that, for example, a directed $3$-cycle has no independent dominating set. 
\begin{figure}
\begin{center}
\begin{tikzpicture}
  \foreach \i in {0,1,2}
     \node[circle, fill = black, scale = .5] (v\i) at ({(360/3 * \i) + 45}:1) {};
     
  \draw[->-, thick] (v1)--(v2);
  \draw[->-, thick] (v2)--(v0);
  \draw[->-, thick] (v0)--(v1);
\end{tikzpicture}
\caption{An example of a IDS-free digraph.}
\end{center}
\end{figure}
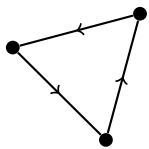
We seek to provide conditions for when a digraph $D$ has an independent dominating set.

In a directed graph $D$ we call a vertex $v$ a \emph{source} if it has $d^-(v) = 0$.
We define the \emph{source-greedy algorithm (SGA)} as follows: for $D$, while there exists a source in the graph, choose one to be placed in the IDS, then remove it and all of its out neighbors. This returns a graph with no sources. 

\begin{claim}
The vertices chosen by the source-greedy algorithm are independent.
\end{claim}

\begin{proof}
We consider the step at which the vertex $v$ is chosen by the SGA.
Since $v$ is chosen, it remains in the graph, so there are no edges from previously chosen sources to $v$.
Also $v$ cannot have had an edge to any previously chosen vertex, else it would not have been a source.
\end{proof}

Note that with the source-greedy algorithm guaranteeing an independent set, we can now refine our search to source-free graphs.

\begin{claim}
All oriented bipartite graphs have an IDS.
\end{claim}

\begin{proof}
First we run SGA.
What remains after SGA is a source-free graph. 
Now we may simply take one side of the bipartition in the independent dominating set.
Note that since there are no sources and the graph is now isolate free, each vertex has at least one in neighbor on the other side, so by taking an entire side, either the vertex is chosen or its in neighbor is chosen.
\end{proof}

\begin{observation}
A graph is IDS-free if and only if every execution of the SGA leaves a source-free graph with no IDS. In particular, every vertex minimal IDS-free graph is source-free.
\end{observation}

\begin{proof}
By contraposition, if there exists an execution which leaves a source-free graph with an IDS, we run that execution and add the remaining IDS.

Suppose the graph has an IDS. Notice that each source must be taken in the independent dominating set, reducing the problem to a subgraph. Repeat.
\end{proof}

A digraph is said to be \emph{acyclic} if it does not contain any subgraphs isomorphic to a directed cycle.

\begin{theorem}\label{thm:DAG}
Every Directed acyclic graph contains an independent dominating set.
\end{theorem}

\begin{proof}
Consider a topological ordering of the vertices. 
The source greedy algorithm will provide an independent dominating set, since at each stage that a vertex set is removed, no cycles are created and we have reduced the problem to another directed acyclic graph.
Since every directed acyclic graph contains a source, this process will only terminate when no vertices remain, namely every vertex was either chosen, or was deleted by a chosen vertex which dominates it.
\end{proof}

\begin{corollary}
Every oriented tree contains an independent dominating set.
\end{corollary}

We now build to the main theorem, expanding the source-greedy algorithm to strongly connected components of a graph. 
Call a graph $G$ \emph{vertex minimal IDS-free} if $D$ has no IDS and for every subset $S \subseteq V(D)$, $D\setminus N^+[S]$ has an IDS $I$ such that $I \cap (N^-(S)) = \emptyset$. 
We define vertex minimal in this way as a generalization of the source-greedy algorithm, where $S$ is acting as a source which can be removed. 

\begin{theorem}
Any vertex minimal IDS-free digraph is strongly connected.
\end{theorem}

\begin{proof}
Let $D$ be a vertex minimal IDS-free digraph. 
Consider the strongly connected components of $G$. 
The reduced graph generated by contracting the strongly connected components is acyclic, hence there exists a source vertex. 
The strongly connected component corresponding to this source vertex, $C$, can be dominated only by other vertices in $C$. 
If $C$ has an IDS, then $G-C\cup(N^+[C])$ has no IDS, else $G$ has an independent dominating set. Otherwise, $C$ has no IDS, a contradiction with minimality of $D$ unless $D=C$. 
\end{proof}

\begin{claim}
Every vertex in a strongly connected digraph has at least one in degree and at least one out degree.
\end{claim}

\begin{proof}
Clear, since if a vertex has $d^+(v) = 0$ there does not exist a path from $v$ to any other vertex, and if $d^-(v)$ there does not exist a path from any other vertex to $v$.
\end{proof}

Since odd cycles are a problem for independent domination, we explore the digraphs with specific periods. 
We define the \emph{period} of a digraph $D$ to be the greatest common divisor among all lengths of directed cycles which appear as subgraphs in $D$. 
As convention, we will say that the period of a directed acyclic graph is $0$. 

We now introduce some tools of linear algebra, which will come in handy for the next proof. 
For a directed graph $D$ on $n$ vertices, we define the \emph{adjacency matrix} of $D$, $A_D$ (or just $A$ if context is clear) to be the $n \times n$ matrix with $(i,j)$ entry $1$ if $(i,j) \in A(D)$ and $0$ otherwise. 
We say that a square matrix is \emph{irreducible} if it is not similar via a permutation matrix to a block upper triangluar matrix. 
The following well known theorem is a fundamental result of spectral graph theory, relating linear algebra and directed graphs:

\begin{theorem}[Godsil and Royle \cite{GodsilRoyle}]
A directed graph $G$ is strongly connected if and only if $A_G$ is irreducible.
\end{theorem}

Perron-Frobenius theory provides a deeper relationship between graph and digraph properties and their respective adjacency matrices. For more information about this relationship, we direct the reader to the textbook \cite{GodsilRoyle}. In particular, the period of a strongly connected digraph creates rich structure in the adjacency matrix, as evidenced by the following theorem.

\begin{theorem}[Frobenius \cite{Frobenius}]
If $G$ is a digraph with period $h>1$, there exists some permutation matrix $P$ such that $PAP^{-1}$ is a block matrix \[PAP^{-1} = \left(
\begin{array}{cccccc}
0      &A_1     & 0      &        &     &0 \\
       &        & A_2    &\ddots  &     & \\
\vdots &        &\ddots  &    \ddots     &     &0  \\
0      &        &        &        &     &A_{h-1} \\
A_h    &0       &        &\cdots  &     &0 \\

\end{array}\right)\]

where each diagonal block is square zero matrix. 
\end{theorem}

We notice that this provides a way to partition our digraph $D$ into $h$ independent sets, which we will call $S_0, \dots S_{h-1}$ corresponding to the vertices of the diagonal zero blocks. With this structure theorem, we may now prove the main theorem of the paper.

\begin{theorem} \label{thm:eper}
Every graph with even period has an independent dominating set.
\end{theorem}

\begin{proof}
If $D$ has period $h$, as above the graph can be partitioned into $h$ independent sets $S_0,\dots,S_{h-1}$ such that there exists an edge from $u$ to $v$ only if $u \in S_i$ and $v \in S_{i+1}$ for some $i \in [h]$ with addition modulo $h$. 
Therefore, we can create an independent dominating set by taking all $S_i$ such that $i$ is even (or odd).

To see that this set is indeed an IDS, since we take only independent sets of the same parity, and $h$ is even there are no parts which are taken that share any adjacencies. 
Furthermore, since every vertex has at least one in degree, we observe one vertex $v$. 
Either $v$ is included in the set, or it is in $S_i$ which is not included and has at least one in degree from $S_{i-1}$, say from $u$. 
But if $S_i$ is not included in our set, $S_{i-1}$ is included in our set, hence $u$ is in the set and dominates $v$.
\end{proof}

\begin{corollary}\label{cor:oper}
If $G$ is vertex minimal IDS-free, $G$ has odd period.
\end{corollary}

\begin{corollary} \label{cor:bipartite}
Every oriented bipartite graph has an independent dominating set.
\end{corollary}

Cary, Cary, and Prabhu\cite{CaryPrabhu} define the maximum number of vertex disjoint independent dominating sets in a digraph $G$ as the \emph{Idomatic Number}, written $id(G)$. 
We note that corollary \ref{cor:bipartite} was proven in this paper as the worked towards determining graphs with $id(G) = 1$.
Our proof provides a bound for the idomatic number of graphs with even period. 

\begin{corollary}
Every strongly connected digraph $D$ with even period has $idom(D) \geq 2$. 
\end{corollary}

\end{section}

\begin{section}{Vizing's Conjecture}
In this section, we show that the analogous statement to the famous Vizing's conjecture does not hold with independent dominating sets. 
Vizing's conjecture is about the relationship between domination number (the smallest size of a dominating set of a graph $G$, $\gamma(G)$) of graphs with their Cartesian product.

We define the Cartesian product of directed graphs with vertex set $V(G) \times V(H)$ with edges defined by :

\[
A(G \square H) = \{ (x,u)(y,v) | xy \in A(G) \text{ and } u = v \text{ or } uv \in A(H) \text{ and } x=y \}
\]

\begin{conjecture}[Vizing \cite{Vizing}]
For any undirected graphs $G$ and $H$, $\gamma(G\square H) \geq \gamma(G)\gamma(H).$
\end{conjecture}

This also has an analogous conjecture in independent domination, asked by Goddard and Henning, which would imply Vizing's conjecture. 
For the \emph{independent domination number}, the smallest size of a dominating set of a graph $G$, denoted $i(G)$.
Vizing's Conjecture is altered in the case of independent domination since it has been proven that there exist graphs $G,H$ such that $i(G\square H) < i(G)i(H)$ \cite{BDGHHK}. 

\begin{conjecture}[\cite{BDGHHK}]
For any undirected graphs $G$ and $H$, \[i(G \square H) \geq \min \{i(G)\gamma(H), \gamma(G) i(H)\}.\]
\end{conjecture}

We will show that the possibility of a directed graph containing no independent dominating set will provide examples that ensure that no such inequality holds in directed graphs.
One may wonder how to define the independent domination number for directed graphs without independent dominating sets.
Some natural candidates for $G$ IDS-free would be $i(G) = 0$, $i(G) = n+1$, or $i(G) = \infty$.
The following corollary shows that the direct translation of the conjecture of Goddard and Henning into directed graphs cannot hold regardless of which convention is chosen.
In the case that $i(G) = 0$ is chosen, Claim \ref{thm:destroyIDS} provides a family counterexamples, and in the other two cases Claim \ref{thm:createids} provides a family of counterexamples.

To provide a family of directed graphs which contain independent dominating sets whose Cartesian product does not contain an independent dominating set we define the following graphs. 
Define $W_n'$ to be a directed wheel on $n+1$ vertices in which the center vertex is dominating and the outside cycle is directed. 
We define $P'$, as an oriented paw with directed edges as in Figure \ref{fig:wheelandpaw}.

\begin{figure}\label{fig:wheelandpaw}
\begin{center}
    \begin{tikzpicture} 
      \foreach \i in {1,2,3}
        \node[draw,black, circle, fill = black, scale = .5] at ({360/3 *\i}:1.2) (v\i) {};
      \node[draw, black, fill = black, circle, scale =.5] at (0,0) (v0) {};
      
      \draw[thick, ->-] (v0) to (v1);
      \draw[thick, ->-] (v0) to (v3);
      \draw[thick, ->-] (v0) to (v2);
      \draw[thick, ->-] (v1) to (v2);
      \draw[thick, ->-] (v2) to (v3);
      \draw[thick, ->-] (v3) to (v1);
      
      \begin{scope}[xshift = 1in]
        \foreach \i in {1,2,3}
          \node[draw,black, circle, fill = black, scale = .5] at ({360/3 *\i}:1) (vp\i) {};
        \node[draw, black, circle, fill=black, scale = .5] at (2,0) (vp0) {};

      \draw[thick, ->-] (vp0) to (vp3); 
      \draw[thick, ->-] (vp1) to (vp2);
      \draw[thick, ->-] (vp2) to (vp3);
      \draw[thick, ->-] (vp3) to (vp1);
        
      \end{scope}
    \end{tikzpicture}
    \caption{The graphs $W_3'$ and $P'$.}
\end{center}
\end{figure}
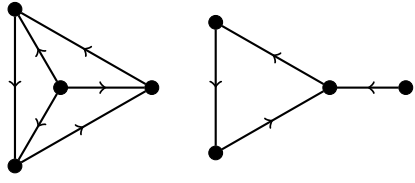

\begin{claim}\label{thm:destroyIDS}
$W_n' \square P'$ is IDS-free for all $n$ odd.
\end{claim}

\begin{proof}
Let $n \in \mathbb{Z}$ be odd.
We notice that both $W_n'$ and $P'$ have unique independent dominating sets by following the source greedy algorithm. 
Let the dominating vertex of $W_n'$ be $v_d$
Also, in $W_n' \square P'$, the copy of $P'$ that appears in place of the dominating vertex of $W_n'$ must be dominated only by vertices of the form $(v_d, u)$ for some $u \in P'$. 
Hence the unique dominating set of $P'$ must be chosen for this copy of $P'$.
Therefore, all copies of the dominating set of $P'$ around the cycle cannot be included in an independent dominating set.

We now look at the strongly connected components of the graph induced by the vertices which have yet to be dominated.
There are two strongly connected components, both of which are cycles on $n$ vertices.
One of these vertices acts as a source in the directed acyclic graph created by contracting strongly connected components, hence it must be dominated only by vertices in its own strongly connected component.
This is impossible, since it is an odd cycle which is known to have no independent dominating set.
\end{proof}

\begin{claim}\label{thm:createids}
$C_n \square C_n$ where $n$ is odd contains an independent dominating set.
\end{claim}

\begin{proof}
Let $n$ be odd.
It has been observed that each directed odd cycle does not have an independent dominating set.
It remains to provide an independent dominating set for $C_n \square C_n$.
Label the vertices of one $V(C_n) = \{v_0,\dots,v_{n-1}$ such that $(v_i, v_{i+1}) \in A(G)$ with addition modulo $n$, and for the other copy of $C_n$, $V(C_n) = \{u_0,\dots,u_{n+1}\}$ similarly.
We construct an independent dominating set of $C_n \square C_n$ as $D = \{(v_i,u_{i+2j}) | 0\leq i \leq n-1 \mod{n}, 0 \leq j \leq \lfloor \frac{n}{2}\rfloor\}$.
To see that the set is dominating, we notice that for any $0 \leq i \leq 4$, $(v_i,u_i)$ and $(v_i,u_{i+2})$ dominate all $(v_i,u_k)$ for $i \leq k \leq i+n-2\rfloor \mod{n}$, leaving only $(v_i,u_{i-1})$ not dominated.
But we have that $(v_{i-1}, u_{i-1}) \in D$ which dominates $(v_{i}, u_{i-1})$.
Therefore $D$ is dominating.
For $i$ fixed, we have $\{v_i,u_{i+2j})| 0 \leq j \leq \lfloor \frac{n}{2}\rfloor\}$ is independent since edges occur if and only if $(u_k, u_\ell) \in E(C_n)$, but we have only taken vertices of the same parity, without taking a full trip around the vertex set.
That is, in each $v_i$ we take only vertices $(v_i,u_j)$ where $i = j \mod 2$.
Hence, vertices $(v_i,u_j)$ and $(v_{i+1},u_k)$ we have no edges, since $j \neq k$.
Therefore $D$ is independent, thus an independent dominating set.
\end{proof}

\begin{theorem}
There exist infinitely many pairs of graphs $(G,H)$ such that \[i(G\square H) > \min { i(G)\gamma(H), \gamma(G)i(H)}\] and infinitely many pairs of graphs $(G',H')$ such that \[i(G'\square H') < \min { i(G')\gamma(H'), \gamma(G')i(H')}.\]
\end{theorem}

\begin{proof}
This is a direct consequence of Claims \ref{thm:createids} and \ref{thm:destroyIDS}.
\end{proof}

\end{section}

\begin{section}{Time Complexity}

We note that finding the size of an independent dominating set in undirected graphs is a well known $NP$-complete problem, for example it is proven in the textbook of Garey and Johnson \cite{GareyJohnson}.

\begin{theorem}[Garey and Johnson \cite{GareyJohnson}]
Given a graph $G$ and constant $k$, determining existence of an independent dominating set $S$ such that $|S| \leq k$ is $NP$-complete.
\end{theorem}

\begin{corollary}
Given a directed graph $D$ and constant $k$, determining existence of an independent dominating set $S$ such that $|S| \leq k$ is $NP$-complete.
\end{corollary}

\begin{proof}
Suppose that we have some oracle $f$ for directed independent dominating sets. 
For $G$, an undirected graph, we may create a corresponding directed graph by replacing every edge with a pair of antiparallel edges, creating a graph $G'$. 
We run $f$ on $G'$. 
By returning whichever result comes from running $f$ on $G'$, we have answered the problem for the undirected graph $G$. 
We see this, since $S$ is an independent set in $G$ if and only if $S$ is an independent set in $G'$ by construction. 
Also $S$ is an dominating set in $G$ if and only if $S$ is an dominating set in $G'$ by construction.
\end{proof}

We notice that the existence of an independent dominating set in a graph $G$ of order $n$ is equivalent to determining if there exists an independent dominating set of order at most $n$. 
In particular, this problem is trivial for undirected graphs since all graphs contain an independent dominating set. 
We seek to determine if for directed graphs determining the existence of an independent dominating set $S$ such that $|S| \leq n$ is $NP$-complete.

\begin{claim}
Given a directed acyclic graph $D$, determining the existence of an independent is in $P$.
\end{claim}

\begin{proof}
By the proof of \ref{thm:DAG}, we provide an algorithm that is polynomial in time.
\end{proof}

\begin{theorem}
Given a directed graph $D$ with even period $h$, determining the existence of an independent dominating set is in $P$.
\end{theorem}

\begin{proof}
Since we know the period of $D$ is $h$, we can construct $h$ independent sets using breadth first search by creating layers modulo $h$ (that is, the $h^\text{th}$ layer is the same as the first vertex chosen in $O(n^2)$ time. Then we select all vertices in even layers as our independent set.
\end{proof}

\begin{question}
Is it true that given a directed graph $D$, determining the existence of an independent dominating set is in $P$? 
\end{question}


In the case that the question above has a negative answer, we provide an algorithm produced by the proof of Theorem \ref{thm:eper}, which gives an exponential time algorithm superior to the brute force algorithm in the case that the period of the digraph is not $1$. 
We note that the period of a digraph can be determined in polynomial time, as proven by Jarvis and Sheir \cite{JarvisShier}.

\begin{theorem}
There exists an $O(2^{\frac{n}{h}})$ algorithm for determining the existence of an independent dominating set in a graph $D$ of period $h$.
\end{theorem}

\begin{proof}
Similar to the above proof, we may partition the vertices of $D$ into $h$ independent sets $S_0,\dots,S_{h-1}$ such that edges follow cyclically. 
We note now that if $h$ is even, we have an independent dominating set, so we may assume that $h$ is odd. 

Let $S_k$ be the smallest independent set in our partition of the vertices, then $|S_k| \leq \frac{n}{h}$. 
We notice now that the selection of vertices in one part forces the structure of the rest of the independent dominating set. 
That is, let $D$ be an independent dominating set of $G$, then $D$ is the union of $S_i \cap D$, $S_{i+1} - N^+(S_i \cap D)$, $S_{i+2} - N^+(S_{i+1}-N^+(S_i \cap D))$, $dots$. 
Note that this observation gives us that $S_{i-1} - (S_{i-1} \cap D) \subseteq N^-(S_i \cap D)$. 
In particular, we may search among only the smallest independent set for the independent set giving the desired bound. 
Since there are $h$ parts, there exists at least one part of size at most $n/h$, and a brute force search among each of the subsets of these vertices will be $O(2^{\frac{n}{h}})$ time.
\end{proof}

\begin{corollary}
For any digraph $D$ with period $h \neq 1$, there exists an $O(1.26^n)$ algorithm to determine existence of an independent dominating set.
\end{corollary}

\begin{proof}
Then the slowest algorithm provided in the proof above for odd degree is $h=3$, yielding an $O(2^{\frac{n}{3}}) \leq O(1.26^n)$ algorithm, since directed acyclic graphs and graphs with even period are in $P$.
\end{proof}

\end{section}

\begin{section}{Constructions}
One may wonder if all graphs with odd period have no independent dominating set. 
We now provide examples for each odd period of infinite families of graphs which have independent dominating sets and which do not have independent dominating sets. 
We start with a few lemmas to work toward constructions of infinite families of graphs with specific period that contain independent dominating sets, and that do not contain independent dominating sets. 

\begin{lemma}\label{lem:nonem}
Let $D$ be a digraph with odd period $h$ and vertex partition $S_1,\dots,S_h$ such that for every edge $(u,v) \in A(D)$, $u \in S_i$ and $v \in S_{i+1}$ for some $i \in [h]$ with addition modulo $h$. Every independent dominating set $I$ has $S_i \cap I \neq \emptyset$ and $S_i \cap I \neq S_i$ for all $i \in [h]$.
\end{lemma}

\begin{proof}
Let $D$ a digraph with odd period $h$ and $S_i$ be as in the statement of the theorem for $1 \leq i \leq h$. If $h=1$, the statement is clear.

Suppose $h > 1$ and assume for contradiction that there exists some $S_i$ such that $S_i \cap I = \emptyset$. 
We notice that the only vertices which can dominate the vertices of $S_{i+1}$ are in $S_i$ or the vertices themselves. 
Therefore, $S_{i+1} \cap I = S_{i+1}$. 
Since the digraph is strongly connected, every vertex in $S_{i+2}$ has a neighbor in $S_{i+1}$, hence $S_{i+2} \cap I = \emptyset$. 
By a similar argument, we see that $S_{i + 2l} \cap I = \emptyset$ for all $l$, with addition modulo $h$. 
Since $h$ is odd, for each $1\leq j\leq h$ there exists some $k$ such that $S_j = S_{i + 2k}$.
Therefore $S_i \cap I = \emptyset$ for all $1 \leq i \leq h$, a contradiction with $I$ being an independent dominating set. 
The argument that $S_i \cap I \neq S_i$ is similar.
\end{proof}

This lemma gives us a simple way to create infinite families of graphs which do not contain independent dominating sets for each period. 
Namely, any strongly connected digraph with odd period $h$ in which the decomposition into $h$ independent sets has at least one set of size $1$ cannot have an independent dominating set. 
We seek to find a family more rich in structure which has no independent dominating set, which will lead to a very similar family that does contain independent dominating sets.

\begin{lemma} \label{thm:const}
For each odd integer $h>1$, there exists an infinite family of graphs $\mathcal{F}$ with period $h$ such that for all $D \in \mathcal{F}$, $D$ is independent dominating set-free. 
\end{lemma}

\begin{proof}
We will construct a graph $D_{h,k}$ with period $h$ for any $2 \leq k$ which has no independent dominating set. 
We will use the fact that since $D$ has period $h$, it can be partitioned into $h$ independent sets $S_0,\dots,S_{h-1}$ such that for all edges $(u,v)$, $u \in S_i$ and $v \in S_{i+1}$ for some $0 \leq i \leq h-1$ with addition modulo $h$. 
We will create $S_0,\dots,S_{h-1}$ as such a partition. 
Let $k \geq 2$.

We create a special graph for the case $h = 3$. 
Let the vertices of $S_0$ be $k$ vertices labelled $1$ to $k$, the vertices of $S_1$ be all subsets of $[k]$, and the vertices of $S_2$ be a copy of the vertices of $S_1$. We draw edges between $u \in S_0$ and $X \in S_1$ if and only if $u \in X$, between $X \in S_1$ and $Y \in S_2$ if and only if $X = Y$, and between $Y \in S_2$ and $v \in S_0$ if and only if $v \in Y$. 

Let $h > 3$. 
Let $S_0$ be $k$ vertices, labelled $1$ to $k$. Define $S_1$ to be all nonempty and not full subsets of $[k]$. 
With edges from $u \in S_0$ to $X \in S_1$ if and only if $u \in X$. Then $S_2$ is a copy of $S_1$ with edges from $X \in S_1$ to $Y \in S_2$ if and only if $X=Y$. $S_3$ will have $k$ vertices, again labelled from $1$ to $k$, with edges $(Y,v)$ from $S_2$ to $S_3$ if and only if $v \notin Y$.
For each $j \leq h-3$ odd, the vertex set of $S_j$ is $k$ vertices labelled $1$ to $k$, and $S_{j+1}$ will have vertices corresponding to subsets of $[k]$ with edges from $\ell \in S_j$ to $Z \in S_{j+1}$ if and only if $\ell \notin Z$ and edges $Z \in S_{j+1}$ to $m \in S_{j+2}$ if and only if $m \notin Z$. For the final independent sets, we follow that $S_{h-2}$ is a set of size $k$ labeled from $1$ to $k$, create $S_{h-1}$ as all subsets of $[k]$, with have edges from $u \in S_{h-2}$ to $X \in S_{h-1}$ if and only if $u \in X$, and finally from $Y \in S_{h-1}$ to $v \in S_0$ if and only if $v \notin Y$. 
See Figure \ref{fig:g53} for an example of $D_{5,3}$.

For any independent dominating set $I$, we claim that $|S_0 \cap I| = 1$. 
Suppose for contradiction that $|S_0 \cap I|  \geq 2$ without loss of generality we may assume that $S_0 \cap I \supseteq \{1,2\}$. Then $S_1 \cap I$ is contains all sets which contain neither $1$ nor $2$. 
Then we have that $S_3 \cap D$ is all sets which contain either $1$ or $2$, in particular, both the set $1$ and $2$ are in the dominating set, and in $S_4$ ${1}$, and ${2}$ dominate $S_5$ since $1$ does not contain $2$ and $2$ does not contain $1$. 
Therefore $S_5 \cap D = \emptyset$. A contradiction with Lemma \ref{lem:nonem}. Indeed, $|S_0 \cap D| = 1$.

Since all vertices of $S_0$ are the same up to isomorphism, and every independent set must have nonempty intersection with the dominating set, we may assume that $S_0 \cap I = {1}$. 
Therefore in $S_1$, only vertices not containing $1$ can be in the dominating set.
Hence in every vertex in $S_3 \cap I$ contains a $1$. 
Therefore $S_4 \cap D$ must contain $1$, and $S_5 \cap D$ must contain only vertices which have a $1$. 
So $S_4 \cap D = S_{4 + 2j} \cap D = 1$ and $S_3 \cap D = S_{3+2j} \cap D$ is all subsets which contain $1$ for all $j$ such that $3 + 2j \leq h-2$. 
At $S_{h-2}$ we have edges from a $k$ set to subsets by inclusion, hence $S_{h-1} \cap D$ is all subsets not containing $1$. 
But these subsets all point to $1$ which is assumed to be in the set, a contradiction with $D$ being independent. 
Therefore no independent dominating set exists.
\end{proof}

\begin{center}
    \begin{figure}
        \label{fig:g53}
        \centering
        \begin{tikzpicture}[scale = .6]
            \node[rectangle, draw, rounded corners, label={$S_2$}] at ({360/5 * 2}:3.5) (S2) {$\emptyset \neq X \subsetneq \{1,2,3\}$};
            \node[rectangle, draw, rounded corners, label=below:{$S_3$}] at ({360/5 * 3}:3.5) (S3) {$\emptyset \neq Y \subsetneq \{1,2,3\}$};
            \node[rectangle, draw, rounded corners, label=right:{$S_5$}] at ({360/5 * 5}:3.5) (S5) {$\emptyset \neq Z \subsetneq \{1,2,3\}$};
    
            \node[circle, draw, label={$S_1$}, minimum size=1.5cm] at ({360/5 * 1}:4) (S1) {};
            \node[circle, draw, label=below:{$S_4$}, minimum size=1.5cm] at ({360/5 * 4}:4) (S4) {};
                \node[circle, scale = .5, draw, fill=black, label=below:{1}] at ({360/5}:3) (v11) {};
                \node[circle, scale = .5, draw, fill=black, label=below:{2}] at ({360/5}:4) (v12) {};
                \node[circle, scale = .5, draw, fill=black, label=below:{3}] at ({360/5}:5) (v13) {};

            \foreach \i in {1,2,3}
                \node[circle, scale = .5, draw, fill=black, label = 1] at ({360/5 * 4}:3) (v41) {};
                \node[circle, scale = .5, draw, fill=black, label = 2] at ({360/5 * 4}:4) (v42) {};
                \node[circle, scale = .5, draw, fill=black, label = 3] at ({360/5 * 4}:5) (v43) {};
%
        \draw[line width = .2cm, ->-, color=black!50] (S1) -- (S2) node[midway, above = .3cm, color=black] {$i \notin X$};
        \draw[line width = .2cm, ->-,  color=black!50] (S2) -- (S3) node[midway, left = .3cm, color=black] {$X = Y$};
        \draw[line width = .2cm, ->-,  color=black!50] (S3)-- (S4) node[midway, below = .3cm, color=black] {$i \notin Y$};
        \draw[line width = .2cm, ->-, color=black!50] (S4) -- (S5) node[midway, right = .3cm, color=black] {$i \in Z$};
        \draw[line width = .2cm, ->-, label={not included}, color=black!50] (S5) -- (S1) node[midway, right = .3cm, color=black] {$i \notin Z$};
        \end{tikzpicture}
        \caption{The digraph $D_{5,3}$.}
    \end{figure}
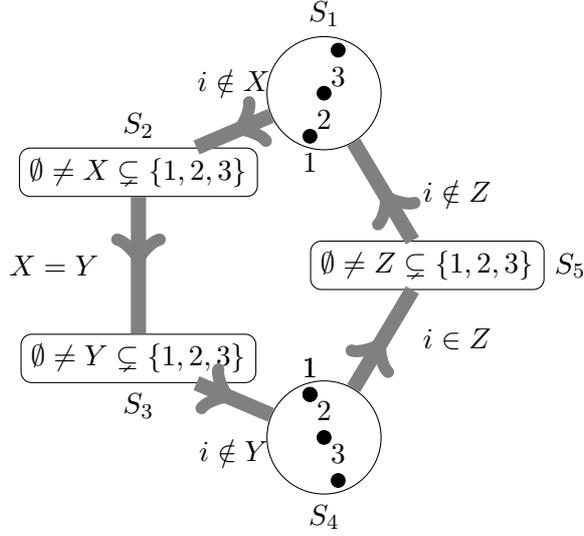
\end{center}

By altering this construction slightly, we instead get a nontrivial family of graphs with odd period which contain independent dominating sets. 
This shows that only knowing the period of a graph is not sufficient for determining existence of an independent dominating set.

\begin{lemma}
In a directed graph $D$ with odd period $h$ and decomposition into independent sets $S_0,\dots,S_{h-1}$ such that vertices in $S_i$ are adjacent only to vertices in $S_{i+1}$ with addition modulo $h$, an independent dominating set $I$ is defined entirely by $S_i \cap I$ for any $0\leq i \leq h-1$.
\end{lemma}

\begin{proof}
Suppose that we have a digraph $D$ with period $h$ decomposed as in the statement, and we have $S_i \cap I = X$ for some $X \subseteq V(D)$.
Notice that the only vertices which can dominate $S_{i+1}$ are vertices of $S_i$ or vertices of $S_{i+1}$.
Therefore, any vertex in $S_{i+1} \setminus N^+(X) \in I$. 
Hence we have determined $S_{i+1} \cap I = S_{i+1} \setminus N^+(X)$.
By the same argument we can now construct $S_{i+2} \cap I$, and taking one step at a time $S_{i+k}$ for any $1\leq k$. 
\end{proof}

\begin{theorem}
For each $h$, there exists an infinite family of graphs $\mathcal{G}$ with period $h$ such that for all $G \in \mathcal{G}$, $G$ has an independent dominating set.
\end{theorem}

\begin{proof}
We follow the construction in \ref{thm:const}, but instead draw edges by from $u \in S_{h-2}$ to $S_{h-1}$ if and only if $u \notin S$. 
We note then that the independent dominating set defined by $S_0 \cap D = 1$ is an independent dominating set. 
In particular, we see that from $S_3$ onward, we alternate $S_i \cap D$ between the set $1$, and the set of all subsets not including one based on parity. 
\end{proof}

\end{section}

\begin{section}{Conclusion and Further Questions}
In this paper, we expanded on independent domination theory in directed graphs by providing a generalization of several of Cary, Cary, and Prabhu's original results, by showing that directed acyclic graphs with even period have independent dominating sets and allowing anti-parallel edges. 
We then initialize the study of time complexity of independent domination in directed graphs.
We note that determining the independent domination number of a directed graph is $NP$-complete, but focus on the question of existence of an independent domination set. 
We prove that for certain classes of graphs, the existence of independent dominating sets is in $P$, and provide an exponential time algorithm for the class of graphs with odd period greater than $1$.
We finally provide constructions of graphs that show that the directed analogue of Vizing's Conjecture for independent dominating sets does not hold. 

There are many significant questions which arise from this paper.
In the area of complexity, since the existence of independent dominating sets in graphs is trivial, it would be interesting to determine whether or not existence in digraphs is in $P$.
Additionally, providing a classification for graphs which contain no independent dominating set, or a proof that the question is $NP$-complete would be of significant interest since it would provide an seemingly difficult digraph question whose analogous graph question is trivial.
Determining the difficulty of classification of graphs with a fixed idomatic number is a natural first place to expand our study.
We wonder also under what restrictions an analogue of Vizing's Conjecture that might hold, for example forcing that all graphs and their cartesian products contain independent dominating sets.
As Cary, Cary, and Prabhu suggest, studying how the reversal or addition of a single edge can alter the idomatic number, which is of interest because of the application of independent domination in ad-hoc networks.

\end{section}

\bibliographystyle{plain}

\bibliography{diids}

\end{document}